\documentclass[a4paper,12pt,draft]{amsart}
\usepackage{eucal}
\usepackage{amsmath,amsthm,amssymb}
\usepackage{amsfonts}
\usepackage{latexsym}
\usepackage[dvips]{graphics}
\usepackage{amsrefs}
\usepackage{eucal}
\usepackage{mathrsfs}
\theoremstyle{plain}
\newtheorem{thm}{Theorem}[section]
\newtheorem{prop}[thm]{Proposition}
\newtheorem{lemma}[thm]{Lemma}
\newtheorem{cor}[thm]{Corollary}

\theoremstyle{definition}
\newtheorem{defn}[thm]{Definition}

\theoremstyle{remark}
\newtheorem{rem}[thm]{Remark}

\numberwithin{equation}{section}

\title{A representation for algebraic $K$-theory of
quasi-coherent modules over affine spectral schemes} 
\date{\today} 
\author{Mariko Ohara}
\address{Department of Mathematical Sciences, Shinshu University, 3-1-1, Asahi, Matsumoto City, JAPAN, 390-8621. 
}  
\email{primarydecomposition@gmail.com}
\newcommand{\Spec}{\mathrm{Spec}\,}
\newcommand{\Hom}{\mathrm{Hom}}
\newcommand{\End}{\mathrm{End}}
\newcommand{\sSet}{\mathrm{Set}_{\Delta}}

\newcommand{\Cat}[1]{\mathrm{Cat}_{#1}}
\newcommand{\Sym}{\mathrm{Sym}}
\newcommand{\Fin}{\mathcal{F}\mathrm{in}_*}

\newcommand{\Map}{\mathrm{Map}}

\newcommand{\Mod}{\mathrm{Mod}}

\newcommand{\CAlg}{\mathrm{CAlg}}

\newcommand{\Fun}{\mathrm{Fun}}

\newcommand{\N}{\mathrm{N}}

\newcommand{\spasce}{\operatorname{Sp}}

\newcommand{\Ind}{\mathrm{Ind}}
\newcommand{\Shv}{\mathcal{S}hv}

\newcommand{\E}{\mathbb{E}_{\infty}}

\usepackage{enumerate}
\usepackage{array}
\usepackage[all]{xy} 
\usepackage{delarray}
\ifx\undefined\bysame
\newcommand{\bysame}{\leavemode\hbox to3em{\hrulefill}\,}
\fi
\begin{document}
\thispagestyle{empty}
\begin{abstract}
In this paper, we study $K$-theory
 of spectral schemes by using locally free sheaves. Let us regard the
 $K$-theory as a functor $K$ on affine spectral schemes. Then, we prove
 that the group completion $\Omega B^{\mathcal{G}} (B^{\mathcal{G}}GL)$
 represents the sheafification of $K$ with respect to Zariski
 (resp. Nisnevich) topology $\mathcal{G}$, where $B^{\mathcal{G}}GL$ is a classifying space of a colimit of affine
 spectral schemes $GL_n$. 
\end{abstract}

\maketitle

\section{Introduction}
An $\infty$-category is a notion of categories up to coherent homotopy. 
The spectral algebraic geometry in terms of $\infty$-category, introduced by
Lurie \cite{DAG5} \cite{DAG7}, is a generalization of algebraic geometry.

In this paper, we study the $K$-theory on spectral schemes. 
For the $K$-theory of a category of
projective modules of finite rank in the sense of spectral algebraic
geometry, we construct an object in the
$\infty$-category which represents the
$K$-theory. In Section (1.1), we explain the
representation of the $K$-theory in the classical algebraic geometry
proved by Morel and Voevodsky~\cite{MV} which motivates us to prove
our main theorem (Theorem~\ref{main1}). 

\subsection{Background of this paper}
Morel and Voevodsky showed the following result of
the $K$-theory in algebraic geometry.

Let $S$ be an affine Noetherian scheme of finite Krull dimension. 
Let $\mathcal{H}(S)$ be a category defined in \cite[Definition 2.1]{MV},
which is arising from certain simplicial
sheaves on a category of smooth schemes of finite type over $S$
equipped with Nisnevich topology. 
\begin{thm}[\cite{MV} Proposition 3.9]\label{mv39}
For the ordinary affine group schemes $GL_n$ for $n \ge 0$, let
 $B(\coprod_n BGL_n)$ be a classifing sheaf in $\mathcal{H}(S)$. 
Let $\Sigma_s$ 
be a simplicial suspension functor on $\mathcal{H}(S)$
 and $R\Omega$ its right adjoint. Then, for a
quasi-projective smooth scheme $X$ over $S$, there is a canonical isomorphism
\[
 \Hom_{\mathcal{H}(S)}(\Sigma^k_s(X), (R\Omega)B(\coprod_n BGL_n)) \cong
 \pi_k K(X), 
\]
where $K(X)$ is the $K$-theory of the ordinary category of locally free sheaves on $X$. In other
 words, $K(X)$ is represented by the object $(R\Omega)B(\coprod_n
 BGL_n))$ in the category $\mathcal{H}(S)$.  
\end{thm}

Since a simplicial presheaf is a model of an $\infty$-categorical
sheaf, this theorem may be interpreted in terms of spectral algebraic geometry. We generalize this result to a certain $\infty$-category of sheaves as follows. 

By an $\infty$-category with $w^{\infty}$-cofibrations we mean a pointed
$\infty$-category with a class of morphisms (called $w^{\infty}$-cofibrations)
satisfying certain conditions (see Definition~\ref{wcof} for details).  
Let $\mathcal{C}$ be an $\infty$-category with $w^{\infty}$-cofibrations. We define the algebraic $K$-theory space of $\mathcal{C}$ by
\[
 K(\mathcal{C}) = \Omega |S_{\bullet}(\mathcal{C})|,
\]
where $S_{\bullet}$ is the $S$-construction (cf. \cite{BGT-End}) 
 and $|-|$ is the geometric realization. 

Let $R$ be a connective $\E$-ring, and $\CAlg^{cn}$ the
$\infty$-category of connective $\E$-rings. 
Let $\Mod_R^{\infty proj}$ be an
$\infty$-category of projective $R$-modules of finite rank which we recall in
Definition~\ref{p0}. It becomes an $\infty$-category with
$w^{\infty}$-cofibrations by Definition~\ref{ssplit}.

We denote by $\CAlg^{\mathcal{G}}$ an
$\infty$-category $\CAlg^{cn}$ equipped with either the Zariski topology
or the Nisnevich topology, and by $\Spec^{\mathcal{G}} R$ an object
in the essential image of Yoneda functor $\CAlg^{\mathcal{G}} \to
\Shv(\CAlg^{\mathcal{G}})$, where $\Shv(\CAlg^{\mathcal{G}})$ is the $\infty$-category of sheaves on
$\CAlg^{\mathcal{G}}$ given in Section $2$. 

Let $\widehat{\mathcal{S}}$ be the $\infty$-category of not-necessary
small spaces.  
We define a functor 
\begin{equation}\label{func}
 K : (\CAlg^{\mathcal{G}})^{op} \to \widehat{\mathcal{S}}
\end{equation}
which carries a spectral scheme $\Spec^{\mathcal{G}} R$ to
the $K$-theory $K(\Mod_R^{\infty proj})$. Here, we denote by $(-)^{op}$ the
opposite $\infty$-category. 

We denote by $\widetilde{(-)}^{\mathcal{G}} $ the
sheafification from the $\infty$-category of functors on
$(\CAlg^{\mathcal{G}})^{op}$ to $\Shv(\CAlg^{\mathcal{G}})$ which we recall in Definition~\ref{Shff}.

Let $GL_n$ be an object in $\Shv(\CAlg^{\mathcal{G}})$. 
Let $B^{\mathcal{G}}GL = \coprod_{n \in \mathbb{N}} B^{\mathcal{G}}GL_n$ be the coproduct of the
classifying sheaf $B^{\mathcal{G}}GL_n$ of $GL_n$, where
$B^{\mathcal{G}}$ is a functor given by taking classifying sheaf which
we recall in Section $4$. 
Let $\Omega B^{\mathcal{G}}$ be a functor on $\Shv(\CAlg^{\mathcal{G}})$
defined in Definition~\ref{0mg}. We denote by $\Omega B^{\mathcal{G}} (B^{\mathcal{G}}GL)$ the group completion on (cf. Definition~\ref{frpr}).

\begin{thm}[cf. Theorem~\ref{m1}]
\label{main1}
Let $\Map_{\Shv_{\widehat{\mathcal{S}}}(\CAlg^{\mathcal{G}})}(-, -)$
 denote the mapping space of the $\infty$-category
 $\Shv_{\widehat{\mathcal{S}}}(\CAlg^{\mathcal{G}})$ which we recall in
 Definition~\ref{mapp}. 
There is an equivalence 
\[
 \Map_{\Shv_{\widehat{\mathcal{S}}}(\CAlg^{\mathcal{G}})}(\Spec^{\mathcal{G}} R, \,
 \Omega B^{\mathcal{G}} (B^{\mathcal{G}}GL)) \simeq \widetilde{K}^{\mathcal{G}}(\Mod_R^{\infty proj}). 
\]
\end{thm}

\subsection*{Remarks for Theorem~\ref{main1}}
It is known that the sheaf $\widetilde{K}^{\mathcal{G}}$ is representable by an object in
$\Shv_{\widehat{\mathcal{S}}}(\CAlg^{\mathcal{G}})$ by
the brown representability theorem in the sense of $\infty$-category~\cite[Proposition 5.5.2.7]{HT}. 
The importance is that we give a concrete object which
represents $\widetilde{K}^{\mathcal{G}}$.

Morel and Voevodsky used the cofinality theorem in Quillen's $K$-theory
in the proof of Theorem~\ref{mv39}. Since the cofinality theorem is not
established in the $K$-theory of $\infty$-categories, their proof cannot
be applied to our case directly. To avoid this problem, we treat an $\infty$-category $\Mod_R^{\infty proj}$ instead of the
finitely generated projective modules.

\subsection{Outline of this paper}
This paper is organized as follows. 
In Section $2$, we introduce the
terminology of an $\infty$-category endowed with Grothendieck topology
and sheaves on those $\infty$-category. 
In Section $3$, we describe the group completion functor on
$\Shv(\CAlg^{\mathcal{G}})$ explicitly in Proposition~\ref{gpcmp}. 
In Section $4$, we show the correspondence between the value of the
affine group scheme $GL_n$ on $S$ and the automorphisms of $S^n$, where
$S$ is an $\E$-ring in Proposition~\ref{GLgp}. We also demonstrate that
the classifying sheaf $B^{\mathcal{G}}GL_n$ is equivalent to the sheaf
of projective modules of finite rank in Proposition~\ref{bgl}. 
We also define the notion of Zariski connectedness and relate a certain
mapping space of $B^{\mathcal{G}}GL_n$ is equivalent to the sheaf
of projective modules of finite rank in Proposition~\ref{qclf}. 
In Section $5$, we recall the several notion and proposition with
respect to $K$-theory. 
In Section $6$, by using these results, we prove Theorem~\ref{m1}.

\subsubsection*{Acknowledgement}
The author would like to express deeply her thanks to Professor Nobuo Tsuzuki
for his valuable advice and checking this paper. The author would like to express her thanks to Professor Satoshi Mochizuki for his valuable comments on algebraic
$K$-theory, especially the resolution theorem and the cofinality, and
for reading this paper. The author also would like to express her thanks to Professor Yuki Kato for valuable
comments to the author. 

\section{Preliminary}
We fix the universe $\mathbb{U}$ such that $\mathbb{N} \in
\mathbb{U}$. 
We define the $\Cat{\infty}$ by
$\mathbb{U}$-small $\infty$-category, which is locally
$\mathbb{U}$-small.

Although there are a lot of languages of higher category theory, we use
the same notation in Lurie's book~\cite{HT} and papers~\cite{HA}, \cite{DAG5} and \cite{DAG8}. 


Let $\Fun(-, \, -)$ be an $\infty$-category of functors. 
For a pair of functors $f : \mathcal{C} \to \mathcal{D}$ and $g :
 \mathcal{D} \to \mathcal{C}$ between $\infty$-categories, according to
 \cite[Proposition 5.2.2.8]{HT}, we say that
 the functor $f$ is a left adjoint to $g$ (resp. $g$ is a right adjoint to $f$) if there exists a
 unit map $u: id_{\mathcal{C}} \to g \circ f$ given in \cite[Defnition 5.2.2.7]{HT}. 

Let $\CAlg$ be an $\infty$-category of $\E$-rings. We denote by
 $\mathbb{S}$ the initial object in $\CAlg$, which is called the sphere
 spectrum. 
We say that a spectrum $E$ is {\it connective} if $\pi_n E \simeq 0$ for $n <
0$. We denote by $\CAlg^{cn}$ a full $\infty$-subcategory of $\CAlg$
 which consists of connective $\E$-ring. 

For an $\E$-ring $R$, we also have an
$\infty$-category $\Mod_R$ which is called the {\it
$\infty$-category of $R$-module} of $\mathrm{Sp}$~\cite[Section
4.2]{HA}. 
Since the tensor product on the $\infty$-category of spectra is compatible with the geometric realizations~\cite[Corollary 4.8.2.19]{HA}, $\Mod_R$ becomes the
 symmetric monoidal $\infty$-category by \cite[Theorem 4.5.2.1]{HA}. We
 denote by $\otimes_R$ the tensor product on $\Mod_R$. 

Let $R$ be an $\E$-ring and $a \in \pi_0 R$ an element. The localization
 of $R$ with respect to $a$, which is denoted by $R[a^{-1}]$, is an
 $\E$-ring (see \cite[Remark 2.9]{DAG7} and \cite[7.2.4]{HA}). 
\begin{defn}\label{p0}
 Let $R$ be a connective $\E$-ring and $M$ an $R$-module. 
\begin{enumerate}[(i)]
\item We say that $M$ is free of rank $n$ if $M \simeq R^{\oplus n}$. 
\item We say that $M$ is finitely generated projective if it is a
retract of finitely generated free $R$-modules. 
\item We say that $M$ is projective of rank $n$ if it is finitely
      generated projective and we can choose
      elements $x_1, \cdots , x_m \in \pi_0 R$ such that they generate
      the unit ideal and each localization $M[x_i^{-1}]$ is a free
      module of rank $n$ over $R[x_i^{-1}]$. 

We denote by $\Mod_R^{nfree}$ (resp. $\Mod_R^{nproj}$) an
 $\infty$-category of free (resp. projective) $R$-modules of rank $n$. 
We denote by $\Mod_R^{proj}$ the $\infty$-category of
finitely generated projective $R$-modules.  
\end{enumerate}
\end{defn}
\begin{rem}
Let $\Mod_R^{\infty proj}$ be an $\infty$-category of projective
 $R$-modules of finite rank. In general, $\Mod_R^{\infty proj}$ and
 $\Mod_R^{proj}$ is not equivalent. 
\end{rem}

As in \cite{HT}, \cite{HA} and \cite{ba}, to treat certain size of limits and colimits, we need to enlarge the universe and fix the size of universes properly. 

We adopt the axiom of universes which allows us to consider that every
cardinal can be strictly upper bounded by a strongly inaccessible
cardinal. Then, there exists a bijection between strongly inaccessible cardinals and universes, and thus we can take a suitable enlargement of universe which we need. 

By the axiom of universe, there exists an enlargement of universes $\mathbb{U} \in \mathbb{V}$ such that every $\mathbb{U}$-small object is also
$\mathbb{V}$-small, so that $\Cat{\infty}$ becomes a $\mathbb{V}$-small
category. 

In this paper, sometimes we need to treat the $\infty$-category $\CAlg$
as small, and to regard the maximal $\infty$-groupoid
$(\Mod_R)^{\simeq}$ of $\Mod_R$ as a space. We also need to treat the
size of limits and colimits. In these
cases, we enlarge the universe as the following definition and proceed the arguments. 
\begin{defn}\label{big}
Let $\mathcal{S}$ be the $\infty$-category of $\mathbb{U}$-small spaces. 
Throughout
this paper, we enlarge the unverse $\mathbb{U} \in \mathbb{V}$ such that $\Mod_R$ is a $\mathbb{V}$-small $\infty$-category. we use the notation $\widehat{\mathcal{S}}$ for
$\mathcal{S}$ and $\widehat{\Cat{\infty}}$ for $\Cat{\infty}$ after
changing the universe from $\mathbb{U}$ to $\mathbb{W}$. 
\end{defn}
\subsection{Sheaves and spectral affine schemes}
Let $\mathcal{X}$ be an $\infty$-category equipped with the Grothendieck
topology in \cite[Definition 6.2.2.1]{HT} and $\mathcal{C}$ an
$\infty$-category which admits limits. 
We denote by $\Shv_{\mathcal{C}}(\mathcal{X})$ an $\infty$-category of
$\mathcal{C}$-valued sheaves on $\mathcal{X}$ in \cite[Definition 6.2.2.6]{HT}. 
\begin{defn}\label{Shff}
Let $\widetilde{(-)} : \Fun(\mathcal{X}^{op}, \mathcal{S}) \to
\Shv_{\mathcal{S}}(\mathcal{X})$ be an localization defined by \cite[Definition 6.2.2.6]{HT}, which is a left
adjoint of the inclusion. For an object $F$ of $\Fun(\mathcal{X},
\mathcal{S})$, we say that $\widetilde{F}$ is a sheafification of $F$.    
\end{defn}
\begin{rem}[\cite{HT} Construction 6.2.2.9, Remark 6.2.2.12]\label{6229}
For an $\infty$-category $\mathcal{C}$ equipped with a Grothendieck
 topology, and a presheaf $F$ on $\mathcal{C}$, the sheafification
 $\widetilde{F}$ of $F$ is given by the following formula: for any $C
 \in \mathcal{C}$, 
\[
 \widetilde{F}(C) = \mathrm{colim}_{\mathcal{C}_{/C}^{(0)}}
 \mathrm{lim}_{C' \in \mathcal{C}_{/C}^{(0)}} F(C'), 
\]
where the first limit is the limit of the simplicial diagram associated
 to $C' \to C$~\cite[Corollary 6.2.3.5]{HT} and the second colimit is
 taken over the collection of covering sieves on $C$. 
\end{rem}

Let $\mathbb{S}$ be a sphere spectrum. Here is a list of morphisms of
$\E$-rings. 
\begin{itemize}
\item Let $f: A \to B$ be a morphism of $\E$-rings. 
Recall that $f: A \to B$ is said to be flat if $\pi_0 B$ is flat $\pi_0
      A$-module and the underlying map of commutative
rings $\pi_0 A \to \pi_0 B$ induces the isomorphism $(\pi_i A)
\otimes_{\pi_0 A} (\pi_0 B) \simeq \pi_i B$ for every
integer $i$. 
\item Let $f: A \to B$ be a morphism of $\E$-rings. $f: A \to B$ is said to be \'etale if it is flat and the underlying map of commutative
rings $\pi_0 A \to \pi_0 B$ is \'etale. The class of \'etale morphisms on $(\CAlg^{cn})^{op}$ satisfies the
 axiom of admissible morphisms by \cite[Proposition 2.4.17]{DAG5}. 
\item We say that a morphism $f: A \to B$ of $\E$-rings is faithfully flat if
 it is flat morphism and the underlying map of commutative
rings $\pi_0 A \to \pi_0 B$ is faithfully flat. 
By a class of faithfully flat morphisms, Grothendieck topology is
 defined on $\CAlg^{op}$, which is called flat topology~\cite[Definition 5.2, Proposition 5.4]{DAG7}.  
\end{itemize}

Recall the Grothendieck topology given by the admissibility on
covering sieves from \cite[Definition 1.2.1]{DAG5}. 
Let $\mathcal{G}^{Sp}_{Zar}(\mathbb{S})$ be an $\infty$-category with
Grothendieck topology given as \cite[Definition2.10]{DAG7}. 
We denote by $\CAlg^{Zar}$ an $\infty$-category
$\Ind(\mathcal{G}^{Sp}_{Zar}(\mathbb{S})^{op})$ equipped with Zariski topology~\cite[Notation 2.2.6]{DAG5}. 
Moreover, this Zariski geometry is finitary~\cite[Remark 2.2.8]{DAG5} by
the definition. (For the base change assumption, see \cite[Remark
1.2.4]{DAG5}). 

Let $A$ be a connective $\E$-ring. Let $\CAlg_{A}^{et}$ be an $\infty$-category spanned by connective \'etale algebras
over $A$. 
Let $\mathcal{C} \subset (\CAlg_A^{et})^{op}$ 
be a sieve containing $A$. 
We say that $\mathcal{C}$ is a Nisnevich covering sieve on $A$ if it
contains a collection of morphisms $A \to A_a$ such that their underlying
maps of commutative rings $\pi_0 A \to \pi_0 A_a$ determine a Nisnevich
covering defined in \cite[Definition 1.1]{DAG11}. 
Let $A$ be a connective $\E$-ring. We define the admissible morphisms in $(\CAlg^{cn})^{op}$ by the
      morphisms corresponding to \'etale morphisms in $\CAlg^{cn}$, and the collection of admissible morphisms $A \to A_a$ generates a
      covering sieve on $A$ if and only if it is a Nisnevich
covering sieve. Then, it generates a Grothendieck topology on
      $\CAlg^{cn}$. It is called the Nisnevich topology on $\CAlg^{cn}$. 
We denote by $\CAlg^{Nis}$ an $\infty$-category $\CAlg^{cn}$ endowed
 with the Nisnevich topology. 

Let $\CAlg^{\mathcal{G}}$ denote either $\CAlg^{Zar}$ or
$\CAlg^{Nis}$. 
Let us denote an object
in the essential image of Yoneda functor $\CAlg^{\mathcal{G}} \to
\Shv_{\mathcal{S}}(\CAlg^{\mathcal{G}})$ by $\Spec^{\mathcal{G}} R$, and we call the object an affine
spectral scheme.

The following proposition is a special case of \cite[Proposition 5.7]{DAG7}. 
\begin{prop}
\label{57}
Let us consider $\CAlg^{op}$ endowed with the flat topology.  
Then, in the case of
$\infty$-topos $\CAlg^{op}$ with flat topology, a functor $F$ is a sheaf if, it preserves
finite products and for any covering $X \to Y$ in $\CAlg$, 
\[
 F(Y) \to \mathrm{lim}_{\Delta}F(X_{\bullet})
\]
is an equivalence. Here, $X_{\bullet} \to Y$ is a simplicial object
associated to $X \to Y$, and the limit in the right hand side is taken
over its simplicial diagram. 
\end{prop}

Note that a Zariski (resp. Nisnevich) covering sieve is a covering which contains a finite collection of morphisms $B \to B_a$ generating the covering sieve. 
\begin{lemma}\label{cover}
\begin{enumerate}[(i)]
\item A Zariski covering sieve is a Nisnevich covering sieve on $\CAlg^{cn}$. 
\item A Zariski covering sieve and a Nisnevich covering sieve are
      covering sieves on the flat topology on $\CAlg^{cn}$. 
\end{enumerate}
\end{lemma}
\begin{proof}
Let $R$ be a connective $\E$-ring and $f_1, \cdots, f_n$ elements in
 $\pi_0 R$ which generate the unit ideal of $R$. To prove (i), we show
 that each $R \to R[f_i^{-1}]$ is flat and $\{\pi_0 R \to \pi_0 R[f_i^{-1}]\}$ is an ordinary Nisnevich covering of $\pi_0 R$. 
Since each $R \to R[f_i^{-1}]$ is flat by~\cite[Remark 2.9]{DAG7} and
 $\{\pi_0 R \to \pi_0 R[f_i^{-1}]\}$ is an ordinary Zariski covering of
 $\pi_0 R$, so that it is a Nisnevich covering of $\pi_0 R$~\cite[Remark
 1.13]{DAG11}. 

To show (ii), we show that any collection of morphisms $R \to R_a$ such that their underlying
maps of commutative rings $\pi_0 R \to \pi_0 R_a$ determine a Nisnevich
covering, $R \to \prod_{\alpha} R_{\alpha}$ is faithfully flat. 
Note that we can assume that this
 Nisnevich covering is a finite collection of morphisms~\cite[Remark
 1.6]{DAG11}. 

Since each $R \to R_{\alpha}$ is \'etale, it is flat. It follows that $R \to \prod_{\alpha} R_{\alpha}$ is also flat. By \cite[Remark
 1.12]{DAG11},  $\pi_0 R \to \pi_0 (\prod_{\alpha} R_{\alpha}) \cong
 \prod_{\alpha} \pi_0 R_{\alpha}$ is faithfully flat. 
\end{proof}

\section{Group completion in an $\infty$-topos}
\begin{defn}[cf. \cite{HA} Definition 2.4.2.1, Definition 5.2.6.2]\label{mono1}
Let $\mathcal{C}$ be an $\infty$-category with finite products, and
 $\mathcal{O}^{\otimes}$ 
 be the 
$\infty$-operad $\mathcal{A}ss^{\otimes}$ (resp. $\N_{\Delta}(\Fin)$)
 given in \cite[Definition 4.1.1.3]{HA} (resp. given by taking the
 simplicial nerve of \cite[Notation 2.0.0.2]{HA}). 
\begin{enumerate}[(i)]
\item An $\infty$-monoid (resp. commutative $\infty$-monoid) in $\mathcal{C}$ is
 a functor $M : \mathcal{O}^{\otimes} \to \mathcal{C}$ such that the
 morphism $M(\langle n \rangle) \to M(\langle 1 \rangle)$ induced by the inert maps $\sigma^i : \langle n \rangle
 \to \langle 1 \rangle$ induces an equivalence $M(\langle n \rangle) \simeq M(\langle 1 \rangle)^n$. 
\item An $\infty$-monoid is an $\infty$-group if its image in $h\mathcal{C}$ is
 a group object. 
\end{enumerate}
\end{defn}
\begin{rem}
In terminology of \cite{GGN}, a commutative $\infty$-monoid is called an
$\E$-monoid, and a commutative $\infty$-group is called an $\E$-group by \cite[Proposition 1.1]{GGN}. 
\end{rem}
Recall that the concepts of $\infty$-monoid and $\infty$-group are defined in Definition~\ref{mono1}. In this section, we treat the several adjoint functors on
$\infty$-categories, and  
we characterize the group completion functor.

\subsection{The group completion functor}
\begin{defn}\label{gba}
Let $\mathcal{C}$ be an $\infty$-category with finite products as in Definition~\ref{mono1}. 
\begin{enumerate}[(i)]
\item In this section, according to the notation in \cite[Definition 1.2]{GGN}, we denote by $Mon(\mathcal{C})$ an $\infty$-category of commutative
$\infty$-monoids in $\mathcal{C}$ and by
$Gp(\mathcal{C})$ an $\infty$-category of commutative $\infty$-groups
in $\mathcal{C}$~\cite[Proposition 1.1]{GGN}. 
\item We say that a functor $Mon(\mathcal{C}) \to Gp(\mathcal{C})$ is a
      group completion on $\mathcal{C}$ if it is the left adjoint to the
      inclusion functor $Gp(\mathcal{C}) \to Mon(\mathcal{C})$. 
\item Let $\mathcal{S}$ be the $\infty$-category of spaces. As a special case, if we take a symmetric monoidal
      $\infty$-category $\mathcal{S}$, we also use the notations
      $nMon(\mathcal{S})$ and $nGp(\mathcal{S})$ $\infty$-categories of
      not-necessary commutative
$\infty$-monoids and not-neccesary commutative $\infty$-groups
in $\mathcal{C}$ respectively. We also say that a functor $nMon(\mathcal{S}) \to nGp(\mathcal{S})$ is a
      group completion on $\mathcal{S}$ if it is the left adjoint to the
      inclusion functor $nGp(\mathcal{S}) \to nMon(\mathcal{S})$.
\end{enumerate}
By the universal
property of adjoint functors, the group completion functor is uniquely determined up to
equivalence~\cite[Remark 5.2.2.2]{HT}.  
\end{defn}

\begin{defn}[cf. \cite{HA} Lemma 2.4.5.9]\label{2459}
\begin{enumerate}[(i)]
\item We denote by $\Fun^{\prod}(\mathcal{C}, \, \mathcal{D})$ an $\infty$-category of functors which preserve the finite products.
\item We define the product on $\Fun^{\prod}(\mathcal{C}, \, \mathcal{D})$
      by taking the objectwise product in $\mathcal{S}$, i.e., the product induced from the formation of product in $\mathcal{S}$ under the sheafification. 
\end{enumerate}
\end{defn}

\begin{lemma}[\cite{GGN} Lemma 1.6]\label{ggn}
Let $\mathcal{C}$ and $\mathcal{D}$ are $\infty$-categories with finite products. Let $\Fun^{\prod}(\mathcal{C}, \, \mathcal{D})$ be an $\infty$-category of functors which preserve the finite products. Then, we have
\[
 Mon(\Fun^{\prod}(\mathcal{C}, \, \mathcal{D})) \simeq \Fun^{\prod}(\mathcal{C}, \, Mon(\mathcal{D})),
\]
and
\[
 Gp(\Fun^{\prod}(\mathcal{C}, \, \mathcal{D})) \simeq \Fun^{\prod}(\mathcal{C}, \, Gp(\mathcal{D})). 
\]
\end{lemma}
\qed 

\begin{defn}[cf. \cite{HT} 6.1.2.7]\label{B}
Let $\mathcal{C}$ be an $\infty$-topos. For an $\infty$-monoid object $G$, 
there exists a
colimit of the simplicial homotopy diagram by ~\cite[Lemma 5.2.2.6]{HA} 
\[
\xymatrix@1{
\cdots \ar@<-1ex>[r] \ar@<-0.7ex>[r] \ar@<-0.4ex>[r] \ar@<-0.1ex>[r] \ar@<0.2ex>[r] & G \times G   \times G 
\ar@<-0.7ex>[r] \ar@<-0.4ex>[r] \ar@<-0.1ex>[r] \ar@<0.2ex>[r] &
     G \times G 
\ar@<-0.5ex>[r] \ar@<-0.2ex>[r] \ar@<0.1ex>[r]  & G \ar@<-0.3ex>[r] \ar@<0.3ex>[r] & 1 .
} 
\]
where the face map is given by the multiplication on $G$ and the
 degeneracy map is given by the unit morphism. 
We define $BG$ is a colimit of the simplicial homotopy diagram. 
\end{defn}
\begin{rem}\label{mono2}
Let $\mathcal{S}$ be the symmetric cartesian monoidal
$\infty$-category of spaces. In terminology of topology, an object $M$ in $\mathcal{S}$ is grouplike if $\pi_0 M$ is a group object in
$h\mathcal{S}$~\cite[Example 5.2.6.4]{HA}. A grouplike object in
 $\mathcal{S}$ is just an $\infty$-group in $\mathcal{S}$ by Definition~\ref{mono1}. 

If $\mathcal{C}$ is a symmetric cartesian monoidal model category, an
 $\infty$-monoid object can be regarded as an associative algebra object
 by~\cite[Proposition 2.4.2.5]{HA}. 
\end{rem}
\subsection{Characterization of the group completion on $\mathcal{S}$}
Let $\mathcal{S}$ be the $\infty$-category of spaces, and
 $\mathcal{S}_{\ast}$ the pointed $\infty$-category of
 spaces (cf. \cite[Notation 5.2.6.11]{HA}). 

For a space $X$, we define the loop space of $X$ by $\Map(S^1, \, X)$, where $S^1$ is a simplicial circle $\Delta^1 / \Delta^0$. 

Let $\mathcal{S}_{\ast \ge 1}$ denote the full $\infty$-subcategory of
      connected spaces in $\mathcal{S}_{\ast}$. There is an equivalence from the $\infty$-category of connected
      spaces to the $\infty$-category of not-neccesary commutative
      $\infty$-groups in $\mathcal{S}$ obtained by the adjoint functors
      $B : nMon(\mathcal{S}) \rightleftarrows \mathcal{S}_{\ast} :
      \Omega$. There is an equivalence from the $\infty$-category $\spasce^{cn}$
      to $Gp(\mathcal{S})$ obtained by $B$ and $\Omega$, which is shown in \cite[Remark 5.2.6.26]{HA}.


\begin{lemma}\label{yy}
Let $\mathcal{S}$ be an $\infty$-category of spaces. 
Let $i: nGp(\mathcal{S}) \to nMon(\mathcal{S})$ be the forgetfull
 functor. 
With the previous
 notations, we have the adjunction
\[
 \Omega B : nMon(\mathcal{S}) \rightleftarrows nGp(\mathcal{S}) : i,
\] 
where $i$ is the forgetfull functor. In other words, $\Omega B$ is the
 group completion on $\mathcal{S}$. 
\end{lemma}
\qed

%
%
%
%
%

\subsection{Characterization of the group completion on $\Fun^{\prod}((\CAlg^{\mathcal{G}})^{op}, \widehat{\mathcal{S}})$}
In this subsection, we
characterize the group completion functor on
$\Fun^{\prod}((\CAlg^{\mathcal{G}})^{op}, \widehat{\mathcal{S}})$ by using
$\Omega$ and $B$. 
\begin{defn}
\begin{enumerate}[(i)]
\item In this paper, we will denote by $\Omega B$ the left adjoint functor given by $\Omega B : nMon(\mathcal{S}) \rightleftarrows nGp(\mathcal{S})$ in Lemma~\ref{yy}. 
\item We also denote by $\Omega B$ the functor
      $\Fun((\CAlg^{\mathcal{G}})^{op}, Mon(\widehat{\mathcal{S}})) \to
      \Fun((\CAlg^{\mathcal{G}})^{op}, \widehat{\mathcal{S}})$ which
      sends a presheaf $F$ to a presheaf given by $R \mapsto \Omega B
      (F(R))$.  
\end{enumerate}
\end{defn}

We will give an explicit description of the group completion functor for the $\infty$-category of
presheaves which preserves finite products. 

\begin{defn}
We define a functor 
\[
 \Omega B^{\prod} :
 Mon(\Fun^{\prod}((\CAlg^{\mathcal{G}})^{op}, \widehat{\mathcal{S}}))
 \to Gp(\Fun^{\prod}((\CAlg^{\mathcal{G}})^{op},
 \widehat{\mathcal{S}})) 
\]
by the restriction of $\Omega B$, i.e., 
$\Omega B^{\prod}$ is the functor which satisfies the following commutative diagram
\[ 
 \xymatrix@1{
Mon(\Fun((\CAlg^{\mathcal{G}})^{op}, \widehat{\mathcal{S}})) \ar[r]^{\Omega B} &
 Gp(\Fun((\CAlg^{\mathcal{G}})^{op}, \widehat{\mathcal{S}}))  \\
Mon(\Fun^{\prod}((\CAlg^{\mathcal{G}})^{op}, \widehat{\mathcal{S}})) \ar[u]
 \ar[r]^{\Omega B^{\prod}} & \ar[u]
 Gp(\Fun^{\prod}((\CAlg^{\mathcal{G}})^{op}, \widehat{\mathcal{S}}))},  
\]
where the vertical morphisms are inclusions. 
\end{defn}

\begin{lemma}
The functor $\Omega B^{\prod}$ gives the group completion on
      $\Fun^{\prod}((\CAlg^{\mathcal{G}})^{op},
      \widehat{\mathcal{S}})$ up to equivalences. 
\end{lemma}
\begin{proof}
Let $\Fun^{\prod}((\CAlg^{\mathcal{G}})^{op}, \widehat{\mathcal{S}}) $ be an $\infty$-category of presheaves on $\CAlg^{\mathcal{G}}$ which preserve the finite products. 
This $\infty$-category has the pointwise finite products. By
 Lemma~\ref{ggn}, we have the equivalences $Mon(\Fun^{\prod}((\CAlg^{\mathcal{G}})^{op}, \widehat{\mathcal{S}})) \simeq  \Fun^{\prod}((\CAlg^{\mathcal{G}})^{op}, Mon(\widehat{\mathcal{S}})) $ and $Gp(\Fun^{\prod}((\CAlg^{\mathcal{G}})^{op}, \widehat{\mathcal{S}})) \simeq  \Fun^{\prod}((\CAlg^{\mathcal{G}})^{op}, Gp(\widehat{\mathcal{S}}))$. 

Therefore, the group completion functor on
$\Fun^{\prod}((\CAlg^{\mathcal{G}})^{op}, \widehat{\mathcal{S}})$ is determined
by the values in the $\infty$-category $\mathcal{S}$ of spaces. The
 problem is reduced to the group completion functor $nMon(\mathcal{S})
 \to nGp(\mathcal{S})$ for the $\infty$-category $\mathcal{S}$ of spaces,
 which is $\Omega B$ by Lemma~\ref{yy}. 
\end{proof}

\subsection{Characterization of the group completion on $\Shv_{\widehat{\mathcal{S}}}(\CAlg^{\mathcal{G}})$}

\begin{defn}[cf. \cite{DAG7} Proposition 1.15]
We define the product on
 $\Shv_{\widehat{\mathcal{S}}}(\CAlg^{\mathcal{G}})$ by the pointwise
 product, i.e., the product induced from the formation of product in
 $\mathcal{S}$ under the sheafification. 
\end{defn}
Since a sheaf preserves finite limits, the pointwise product on
 $\Shv_{\widehat{\mathcal{S}}}(\CAlg^{\mathcal{G}})$ is the restriction
 of the product on $\Fun^{\prod}((\CAlg^{\mathcal{G}})^{op},
 \widehat{\mathcal{S}})$ in Definition~\ref{2459}. 
\begin{lemma}\label{sg}
We have $\Shv_{Gp(\widehat{\mathcal{S}})}(\CAlg^{\mathcal{G}}) \subset Gp(\Shv_{\widehat{\mathcal{S}}}(\CAlg^{\mathcal{G}}))$.
\end{lemma}
\begin{proof}
By definition, the objectwise products of presheaves becomes pointwise products of sheaves after sheafification. 
Since the sheafification is left exact~\cite[Definition 5.3.2.1]{HT}, it commutes with the finite products. Therefore, we have
 $\Shv_{Gp(\widehat{\mathcal{S}})}(\CAlg^{\mathcal{G}}) \subset
 Gp(\Shv_{\widehat{\mathcal{S}}}(\CAlg^{\mathcal{G}}))$.  
\end{proof}

\begin{defn}\label{0mg}
We define the functor 
\[
\Omega B^{\mathcal{G}} :
 Mon(\Shv_{\widehat{\mathcal{S}}}(\CAlg^{\mathcal{G}})) \to
 Gp(\Shv_{\widehat{\mathcal{S}}}(\CAlg^{\mathcal{G}}) 
\]
by the composition of the inclusion $i' :
 Mon(\Shv_{\widehat{\mathcal{S}}}(\CAlg^{\mathcal{G}}))  \to
 Mon(\Fun^{\prod}((\CAlg^{\mathcal{G}})^{op}, \widehat{\mathcal{S}}))$ with 
\[
 \xymatrix@1{
 Mon(\Fun^{\prod}((\CAlg^{\mathcal{G}})^{op}, \widehat{\mathcal{S}}))
 \to Gp(\Fun^{\prod}((\CAlg^{\mathcal{G}})^{op},
 \widehat{\mathcal{S}})) \to Shv_{Gp(\widehat{\mathcal{S}})}(\CAlg^{\mathcal{G}}),
}
\]
where the first functor is the functor which is induced from the pointwise group completion $\Omega B^{\prod} : Mon(\Fun^{\prod}((\CAlg^{\mathcal{G}})^{op}, \widehat{\mathcal{S}}))
 \to Gp(\Fun^{\prod}((\CAlg^{\mathcal{G}})^{op},
 \widehat{\mathcal{S}}))$, and the second functor is obtained by the
 equivalence in Lemma~\ref{ggn} and the
 sheafification functor $\tilde{(-)} : \Fun^{\prod}((\CAlg^{\mathcal{G}})^{op},
 Gp(\widehat{\mathcal{S}})) \to
 Shv_{Gp(\widehat{\mathcal{S}})}(\CAlg^{\mathcal{G}})$. Note that, we
 have $\Shv_{Gp(\widehat{\mathcal{S}})}(\CAlg^{\mathcal{G}}) \subset
 Gp(\Shv_{\widehat{\mathcal{S}}}(\CAlg^{\mathcal{G}}))$ by Lemma~\ref{sg}. 
\end{defn}

\begin{prop}\label{gpcmp}
Let $\CAlg^{\mathcal{G}}$ be the $\infty$-category equipped with the
 Grothendieck topology which is defined in Section 2. 
Then, $\Omega B^{\mathcal{G}}$ is the group completion on
 $\Shv_{\widehat{\mathcal{S}}}(\CAlg^{\mathcal{G}})$. 
\end{prop}
\begin{proof}
By the adjunction in Lemma~\ref{yy} and the definition of $\Omega B^{\mathcal{G}}$, we have 
\[ 
 \xymatrix@1{
\Map(\Omega B^{\mathcal{G}} F, \, G) \simeq \Map((-)^{\simeq} \circ \Omega B \circ i' (F), \, G) \simeq \Map(\Omega B iF, \, iG) \simeq \Map(F, \, iG ). 
} 
\]
\end{proof}

\section{Classifying sheaf of $GL$ and projective modules of finite rank}
\subsection{The affine spectral scheme $GL_n$}
For a simplicial set $S$ and its vertexes $x$ and $y$, recall that a
simplicial set $\Map_S(x, \, y)$ is defined as follows. This construction is due to Joyal. 

\begin{defn}\label{mapp}
[cf. \cite{HT} 1.2.2.2, Corollary 4.2.1.8]\label{mapsp}
Let $S$ be a simplicial set and $S^{\Delta^1}$ a simplicial
 set which sends $[n]$ to $\Hom_{\sSet}(\Delta^1 \times \Delta^n, \,
 S)$. Let $s, t : [0] \to [1]$ be maps defined by $s(0) =0$ and $t(0)=1$. 
\begin{enumerate}[(i)]
\item Take vertexes $x, y \in S$. The mapping space $\Map_S(x,\, y)$ from $x$ to
 $y$ is defined by the pullback 
\[ 
 \xymatrix@1{
\Map_S(x,\, y) \ar[r] \ar[d] & S^{\Delta^1} \ar[d]^{(s,t)} \\
\ast \ar[r]^{(x,y)} & S \times S,
} 
\]
where the morphism $(s,t)$ is induced by $s$ and $t$, e.g., it sends $\Hom_{\sSet}(\{0\} \times \Delta^n, \,
 S)$ to the $n$-simplices $S_n$ of the first factor and $\Hom_{\sSet}(\{1\}
 \times \Delta^n, \, S)$ to the $n$-simplices $S_n$ of the second
 factor, and the morphism $(x,y)$ sends $\ast$ to $(x,y) \in S \times
 S$.  
\item If $S$ is an $\infty$-category, $\Map_{S}(x, \,
y)$ becomes a Kan complex by \cite[Proposition 1.2.2.3]{HT}. 
For an $\infty$-category $S$ and objects $x, y \in S$, we say that $\Map_{S}(x, \,
y)$ is the mapping space between $x$ and $y$.
\end{enumerate}
\end{defn}

\begin{defn}[\cite{HA} Notation 3.1.3.8]\label{259}
Let $R$ be an $\E$-ring, and $\CAlg_R$ the $\infty$-category of
 $R$-algebras. 
\begin{enumerate}[(i)]
\item We define a functor $\mathrm{Sym}_R : \Mod_R \to \CAlg_R$ the left
      adjoint of the forgetful functor $\CAlg_R \to \Mod_R$ which sends an $R$-algebra $S$ to
an $R$-module $S$. 
\item For a free $R$-module $R^{\oplus n^2}$ of rank $n^2$, $\pi_0 \Sym_R R^{\oplus n^2}$ is isomorphic to the polynomial ring $(\pi_0 R)[x_{11},
 \cdots, x_{nn}]$ over $\pi_0 R$. Let us denote $\Sym_R R^{\oplus n^2}$
 by $R\{x_{11}, \cdots, x_{nn}\}$, where each $x_{ij}$ is the
 representative of the indeterminate of the polynomial ring over $\pi_0
 R$. 
\end{enumerate}
\end{defn}
\begin{defn}[\cite{HA} Section 4.7.2]
For an object $X$ in a symmetric monoidal $\infty$-category, an endomorphism object $\End(X)$ is an object equipped with the evaluation morphism $e: \End(X) \otimes X \to X$ which induces a weak homotopy equivalence $\Map(Y \otimes X, \, X) \simeq \Map(Y, \, \End(X))$ for every $Y$. 

Let $R$ be an $\E$-ring. 
We have the endomorphism object of $R^{\oplus n}$ in $\Mod_R$. Let us
 denote it by $\End_R(R^n)$. 
\end{defn}
The following lemma is explained in \cite[Remark 7.1.2.2]{HA}. 
\begin{lemma}[\cite{HA} Remark 7.1.2.2]
We have an isomorphism 
\[
 \pi_{\ast} \End_R(R^{\oplus n}) \cong \pi_{\ast} \Map_R(R^{\oplus n}, \, R^{\oplus n})
\]
\end{lemma}
\qed

The object $R^{\oplus n^2} $ of $\Mod_R$ satisfies the universal property of endomorphism, we have an equivalence $R^{\oplus n^2} \simeq \End_R(R^{\oplus n})$. 
Since we have $R^{\oplus n^2} \simeq \End_R(R^{\oplus n})$, we have an
equivalence $\Sym_R \End_R(R^{\oplus n}) \simeq R\{x_{11}, \cdots,
x_{nn}\}$ by Definition~\ref{259} (ii). 

\begin{defn}
Let $R$ be an $\E$-ring. Let us denote $\Spec^{Zar} \Sym_R
 \End_R(R^{\oplus n})$ by $M_{n, R}$ in $\Shv(\CAlg^{Zar})$. 
\begin{enumerate}[(i)]
\item Under the equivalence $\Sym_R \End_R(R^{\oplus n}) \simeq R\{x_{11}, \cdots, x_{nn}\}$, we define the element $(det) \in \pi_0 M_{n, R}$ by the determinant relation $\Sigma_{\tau \in S_n} sgn(\tau) \prod
 x_{i, \tau(i)}$ of $x_{ij}s$. 
\item We define an affine scheme $GL_{n, R}$
 by inverting the determinant element of $M_{n, R}$. 
\item In the
 case that the base scheme $R$ is the sphere spectrum $\mathbb{S}$, we
      denote $M_{n, R}$ and $GL_{n, \mathbb{S}}$ by $M_n$ and $GL_n$. 
\end{enumerate}
\end{defn}
Let $S$ be an $R$-algebra. We denote by $M_{n, R}(S)$ and $GL_{n, R}(S)$
the mapping spaces $\Map_{\CAlg_R}(\Sym_R \End_R(R^{\oplus n}), S)$ and $\Map_{\CAlg_R}(\Sym_R \End_R(R^{\oplus n})[(det)^{-1}], S)$ respectively.

\begin{rem}\label{juyo}
Since $GL_n$ is corepresented by an $\E$-ring, it is flat sheaf
 by~\cite[VII, Theorem 5.15]{DAG7}, so that it is already a Nisnevich
 sheaf. We also use the notation $GL_n$ for the image in
 $\Shv(\CAlg^{Nis})$ under the sheafification. 
\end{rem}

\subsection{The equivalence $GL_n(R) \simeq Aut(R^n)$ as $\infty$-groups}
\begin{defn}\label{Aut}
We define the space $Aut_R(S^n)$ by the following pullback of simplicial sets
\[ 
 \xymatrix@1{
Aut_R(S^n) \ar[r] \ar[d] & \Map_{\Mod_R}(S^n, S^n) \ar[d]^{\pi_0} \\
 (\pi_0 \Map_{\Mod_R}(S^n, S^n))^{\times} \ar[r] & \pi_0 \Map_{\Mod_R}(S^n, S^n),
} 
\]
where we regard $(\pi_0 \Map_{\Mod_R}(R^n, R^n))^{\times}$ and $\pi_0 \Map_{\Mod_R}(R^n, R^n)$ as constant simplicial sets and $(\pi_0 \Map_{\Mod_R}(R^n, R^n))^{\times}$ is the invertible objects in $\pi_0 \Map_{\Mod_R}(R^n, R^n)$. 
\end{defn}
\begin{rem}\label{e43} 
Note that the constant simplicial sets which appear in the diagram of
 Definition~\ref{Aut} are also Kan complexes since the homotopy set of
 the mapping space is made from the same mapping space by replacing
 $1$-simplices with isomorphism. (Note that any $1$-simplex in a Kan complex is invertible. ) 
\end{rem}
\begin{prop}\label{GLgp}
For an $R$-algebra $S$, we have an equivalence $GL_{n,R}(S) \simeq
 Aut_R(S^n)$ as $\infty$-groups in $\mathcal{S}$, which is functorial
 with respect to $R$-algebra $S$. 
\end{prop}
\begin{proof}
Note that the right vertical morphism $\pi_0$ in the diagram in
 Definition~\ref{Aut} is a Kan fibration. 

Since $GL_{n,R}(S)$ is formulated by the following pullback of simplicial sets
\begin{equation}\label{e42}
 \xymatrix@1{
GL_{n,R}(S) \ar[r] \ar[d] & M_{n, R}(S) \ar[d]^{\pi_0} \\
 (\pi_0 M_{n, R}(S))^{\times} \ar[r] & \pi_0 M_{n, R}(S),
} 
\end{equation}
where we regard $\pi_0 M_{n, R}(S)$ as a constant simplicial set and
 $(\pi_0 M_{n, R}(S))^{\times}$ is the invertible objects in $\pi_0
 M_{n, R}(S)$, and by the coglueing lemma (cf. \cite{HT} A.2.4.3), it
 suffices to construct the two morphisms $M_{n,
 R}(S) \to \Map_{\Mod_R}(S^n, S^n)$ and $\pi_0 M_{n, R}(S) \to \pi_0
 \Map_{\Mod_R}(S^n, S^n)$ which preserve the multiplication and show that the following diagram is commutative:
\begin{equation}\label{e41} 
 \xymatrix@1{
 (\pi_0 M_{n, R}(S))^{\times} \ar[r] \ar[d] & \pi_0 M_{n, R}(S) \ar[d] & M_{n,
 R}(S) \ar[d] \ar[l]^{\pi_0} \\
 (\pi_0 \Map_{\Mod_R}(S^n, S^n))^{\times} \ar[r] &\pi_0
 \Map_{\Mod_R}(S^n, S^n) & \Map_{\Mod_R}(S^n, S^n) \ar[l]^{\pi_0}.  
} 
\end{equation}
Note that objects in the diagram (\ref{e42}) are Kan complexes by the same reason of Remark~\ref{e43}

An element of $M_{n, R}(S)$ is a morphism $\Sym_R \End_R (R^{\oplus n}) \to S$ of $R$-algebras for an $R$-algebra $S$. By the adjointness of $\Sym_R$, the morphism of $R$-algebras is corresponding to a morphism of $R$-modules
\[
 R^{\oplus n^2} \simeq \End_R(R^{\oplus n}) \to S. 
\]
On the other hand, we have an equivalence 
$R^{\oplus n^2} \simeq \End_R(R^{\oplus n})$ obtained by evaluating the each factor of $R^n$. 
Therefore, if we regard $S \otimes_R R^{\oplus n^2}$ as
an $S$-module, the above morphism corresponds to a morphism $S \otimes_R
R^{\oplus n^2} \to S$ of $S$-modules. This gives an identification
$M_{n, R}(S) \simeq \End_S(S^{\oplus n})$ as $S$-modules. 

To show that $GL_{n,R}(S) \to Aut_R(S^n)$ is a morphism of
 $\infty$-groups, we fix the choice $\Map_{\Mod_S}(S^{\oplus n^2}, \, S) \simeq \Map_{\Mod_S}(S, \, S)^{\oplus n^2} \simeq \Map_{\Mod_S}(S^n, \, S^n)$ of the second equivalence. 
Then, by the composition of $S$-module endomorphisms on $S^{\oplus n}$, $\End_S(S^{\oplus n})$ has a canonical $\infty$-monoid structure
for each $S$ defined in Definition~\ref{mono1}, the spectral scheme $M_{n, R}$ is an
$\infty$-monoid. 

We can identify the discrete group $\pi_0 M_{n, R}(S)^{\times}$ with
 the class of $\pi_0 S$-algebra morphisms $\pi_0 \Sym_R \End_R(R^{\oplus
 n})[det^{-1}] \to \pi_0 S$. Here, $det$ is the determinant element given
 by the determinant relation of $x_{ij}s$ in $\pi_0 R[x_{11}, \cdots ,
 x_{nn}]$. By ordinary theory of affine group schemes, these morphisms
 corresponds to $(\pi_0 \End_S(S^{\oplus n}))^{\times}$ by the same
 choice of isomorphism $\Hom(\pi_0 S, \,
 \pi_0 S)^{\oplus n^2} \simeq \Hom(\pi_0 S^n, \, \pi_0 S^n)$. Since
 $2$-morphisms are invertible, this induces an equivalence between
 $Aut_R(S^n)$ and $\Sym_R \End_R(R^{\oplus n})[det^{-1}] \to S$ by the
 above pullback. 

Since we fix the choice of equivalences in the proof, the affine scheme
$GL_{n,R}$ is an $\infty$-group scheme with respect to the monoid structure
of $\Map_{\Mod_R}(R^n, R^n)$. 
\end{proof}

\subsection{The $\widehat{\Cat{\infty}}$-valued functor $(nProj)$}

In this subsection, the functor $(nProj) : (\CAlg^{cn})^{op} \to \widehat{\Cat{\infty}}$ defined by sending an $\E$-ring $R$ to the maximal Kan complex of $\infty$-category $\Mod_R^{\infty proj}$ of projective modules of finite rank is a sheaf. 

\begin{defn}[\cite{DAG8} Definition 2.6.14]\label{2614}
Let $P$ be a property for objects $(A, M)$ in an $\infty$-category $\CAlg^{cn} \times_{\CAlg}
 \Mod$. We say that $P$ is local for the flat topology if the following
 conditions are satisfied: 
\begin{enumerate}[(i)]
\item Let $f: A \to B$ be a flat morphism of connective $\E$-rings, and
      $M$ an $A$-module. If $(A, M)$ has the property $P$, $(B, B
      \otimes_A M)$ has the property $P$. If $f$ is faithfully flat, the
      converse holds. 
\item For any finite collection $(A_i, M_i)$ of the objects in
      $\CAlg^{cn} \times_{\CAlg} \Mod$ such that each $(A_i, M_i)$ has the property
      $P$, the product $(\prod A_i, \prod M_i)$ has the property $P$. 
\end{enumerate}
\end{defn}
\begin{thm}[\cite{DAG7} Corollary 6.13, Lemma 6.17]\label{617}
A functor $(\CAlg^{cn})^{op} \to \widehat{\Cat{\infty}}$ given by $R \mapsto \Mod_R$ is a sheaf.
\end{thm}

\begin{lemma}\label{53}
Let $(nProj) : (\CAlg^{cn})^{op} \to \widehat{\Cat{\infty}}$ be a functor given
 by $R \mapsto \Mod_R^{nproj}$. Then, the functor $(nProj)$ is a sheaf with
 respect to flat topology. 
\end{lemma}
\begin{proof}
Since a functor $(\CAlg^{cn})^{op} \to
 \widehat{\Cat{\infty}}$ given by $R \mapsto \Mod_R$ is already a
 sheaf By Theorem~\ref{617}, it suffices to check that the projective
 modules of rank $n$ satisfy the condition of
 Definition~\ref{2614} (i) and (ii). 

By \cite[Proposition 2.6.15 (1), (6), (9)]{DAG8}, the condition of
 finitely generated projective is flat local property. Since the tensor
 product preserves rank, the projective
 modules of rank $n$ satisfy the condition of
 Definition~\ref{2614} (i). 

We will check the condition $(ii)$ of Definition~\ref{2614}.
Assume that $(A_i, \, M_i)$ is a pair such that $A_i$ is connective
 $\E$-ring and $M_i$ is a projective $A_i$-module of rank $n$ for $1 \le
 i \le m$. 
To show that $\prod_i M_i$ is a projective $\prod_i
 A_i$-module of rank $n$, it suffices to show that there exists a finite
 set $\{g_a\}_{a}$ of objects in $\prod_i A_i$ such that each $\prod_i
 M_i [g_a^{-1}]$ is a free $\prod_i A_i [g_a^{-1}]$-module of rank $n$. 
We choose such $\{g_a\}_a$ as follows. 

For each $i$, we take $f_{i1}, \cdots , f_{ik}
 \in \pi_0 A_i$
 such that $(M_i)[f_{il}^{-1}]$ is a free $A_i [f_{il}^{-1}]$-module of
 rank $n$ for $1 \le l \le k$. 
Since $\prod_i A_i \to A_i$ is flat and the essential image is generated
 by the form $\prod (A_i \otimes_A A_j) \otimes_{A_j} M$ and $A_i
 \otimes_A A_j \simeq 0$, $A_j \otimes_{\prod_i A_i} \prod_i M_i \simeq
 M_j$. 

From this, we have the equivalence $\prod_i A_i[f_{il}^{-1}]
 \otimes_{\prod_i A_i} \prod_i M_i \simeq \prod_i M_{i}[f_{il}^{-1}]$,
 where $\prod_i M_{i}[f_{il}^{-1}]$ is regarded as a $\prod_i A_i
 [f_{il}^{-1}]$-module and is free of rank $n$.  
\end{proof}
\begin{cor}
(nProj) is a Zariski and Nisnevich sheaf. 
\end{cor}
\begin{proof}
It follows from Lemma~\ref{cover} and Lemma~\ref{53}. 
\end{proof}


\begin{defn}[\cite{HT} Theorem 3.1.5.1]
Let $\widehat{\mathcal{S}}$ be the $\infty$-category of spaces and
 $\widehat{\Cat{\infty}}$ the $\infty$-category of
 $\infty$-categories defined in Definition~\ref{big} after enlarging the universe respectively. 
We regard the Kan complexes as $\infty$-categories, so that we have an inclusion 
 functor $i : \widehat{\mathcal{S}} \to \widehat{\Cat{\infty}}$. 

Since the inclusion preserves small colimits, by \cite[Corollary
 5.5.2.9]{HT}, there is an adjunction 
\begin{equation}\label{shf}
i: \widehat{\mathcal{S}} \rightleftarrows \widehat{\Cat{\infty}} (-)^{\simeq}.
\end{equation}
\end{defn}
Note that the right adjoint $(-)^{\simeq}$ is given by taking the maximal $\infty$-groupoid. 


\begin{lemma}
The right adjoint in (\ref{shf}) induces
 a functor from $\widehat{\Cat{\infty}}$-valued sheaves to
$\widehat{\mathcal{S}}$-valued sheaves. 
\end{lemma}
\begin{proof}
Since $(-)^{\simeq}$ is a right adjoint, it preserves limits. Especially, it preserves the limits of simplicial diagrams in Proposition~\ref{57}. Therefore,
$(-)^{\simeq}$ sends $\widehat{\Cat{\infty}}$-valued sheaves to
$\widehat{\mathcal{S}}$-valued sheaves. 
\end{proof}

\subsection{Comparison between $B(GL_n(R))$ to $(\Mod_R^{nfree})^{\simeq}$}
Recall that the notion of classifying object from Definition~\ref{B}. 
We denote by $B^{\mathcal{G}}$ the classifying functor on
 $\Shv_{\widehat{\mathcal{S}}}(\CAlg^{\mathcal{G}})$. 
Acording to Remark~\ref{juyo}, we have $GL_n$ in $\Shv(\CAlg^{\mathcal{G}})$. 
\begin{defn}\label{frpr}
\begin{enumerate}[(i)]
\item We use the notation $B(GL_n(R))$ for the classifying space of the
      value $GL_n(R)$
 in $\mathcal{S}$, and $(B^{\mathcal{G}}GL_n)(R)$ for the value at $R$ of the
 classifying sheaf of $GL_n$ in $\Shv(\CAlg^{\mathcal{G}})$.  
\item Let $(\Mod_{(-)}^{nfree})^{\simeq} : (\CAlg^{cn})^{op} \to \widehat{\mathcal{S}}$ be a functor which sends $R$ to the
 $\infty$-groupoid $(\Mod_R^{nfree})^{\simeq}$ of free
 $R$-modules of rank $n$. 
\item Let $B(GL_n(-)): (\CAlg^{cn})^{op} \to \widehat{\mathcal{S}}$ be a functor given by $R \mapsto B(GL_n(R))$. 
\end{enumerate}
\end{defn}
Note that $GL_n(R) \simeq Aut_R(R^{\oplus n})$, and
$B(GL_n(R))$ is not equal to $(B^{\mathcal{G}}GL_n)(R)$. 

\begin{rem}
Apparently, the classifying sheaf $B^{\mathcal{G}}GL_n$ of $GL_n$ in
 $\Shv(\CAlg^{\mathcal{G}})$ in Definition~\ref{frpr} depends the
 Grothendieck topology on $\CAlg^{\mathcal{G}}$, but consequently, we
 show that it is equivalent to the flat sheaf $(nProj)$ in
 Proposition~\ref{bgl} below, $B^{\mathcal{G}}GL_n$ in Zariski topology is equivalent
 to that in Nisnevich topology.  
\end{rem}

Now, we characterize a functor in Definition~\ref{frpr}(ii) with a
functor in Definition~\ref{frpr}(iii). 

\begin{lemma}\label{aut}
Let us regard $(\Mod_{R}^{nfree})^{\simeq}$ as a simplicial set, and take
 a vertex $x \in (\Mod_{R}^{nfree})^{\simeq}$. 
\begin{enumerate}[(i)]
\item We have a morphism from $Aut_R(x)$ to
      $(\Mod_{R}^{nfree})^{\simeq}$ which is compatible with the face
      maps and the degeneracy maps, i.e., a morphism of simplicial sets. 
\item Under the morphism in (i), we have a morphism $B(Aut_R(x)) \to
      (\Mod_R^{nfree})^{\simeq}$ of simplicial sets.
\end{enumerate}
\end{lemma}
\begin{proof}
If we have $x = y $ in the setting in Definition~\ref{mapp}, we have a morphism
 $\Map(x,\, x) \to X$ of simplicial set by restricting the projection
 $(s,t)$ on the first or the second factors. Thus, we have a morphism from $\Map(x, x)$ to $X$ which is
 compatible with the face maps and the degeneracy maps. 
By composing the projection from $n$-times product $\Map(x, x)^n$ to
 $\Map(x, x)$ with $\Map(x, x) \to X$, we have $\Map(x, x)^n \to X$ which is
 also compatible with the face maps and the degeneracy maps. 

For (i), by applying the above argument to the simplicial set
 $\Mod_{R}^{nfree}$ and a vertex $x \in
 \Mod_{R}^{nfree}$, we obtain the morphisms $\Map_{\Mod_R}(x, \,
 x) \to \Mod_{R}^{nfree}$ and $\Map_{\Mod_R}(x, \,
 x)^n \to \Mod_{R}^{nfree}$ which is compatible with the face
 and degeneracy maps, i.e., morphisms of simplicial sets. 

Recall that $Aut_R(x)$ defined in Definition~\ref{Aut} is a
 simplicial subset of the mappling space $\Map_{\Mod_R}(x, x)$. 
Then, the induced morphisms 
$Aut_R(x) \to \Map_{\Mod_R}(x, \,
 x) \to \Mod_{R}^{nfree}$ factors through the maximal Kan complex $Aut_R(x) \to
 (\Mod_{R}^{nfree})^{\simeq}$. From this, we have the morphism $Aut_R(x)^n \to
 (\Mod_{R}^{nfree})^{\simeq}$ which is compatible with the face
 and degeneracy maps.

For (ii), by the above argument, we have the morphism $Aut_R(x)^n \to
 (\Mod_{R}^{nfree})^{\simeq}$ which is compatible with the face
 and degeneracy maps, i.e., morphism of simplicial set, from the each factor of
 $B(Aut_R(x))$ to $(\Mod_{R}^{nfree})^{\simeq}$. 
Therefore, it induces a morphism $B(Aut_R(x)) \to (\Mod_{R}^{nfree})^{\simeq}$ of simplicial sets. 
\end{proof}

\begin{prop}
A functor in Definition~\ref{frpr}(ii) is characterized by a
functor in Definition~\ref{frpr}(iii). 

Let $(\Mod_{(-)}^{nfree})^{\simeq} : (\CAlg^{cn})^{op} \to \widehat{\mathcal{S}}$ be a functor which sends $R$ to the
 $\infty$-groupoid $(\Mod_R^{nfree})^{\simeq}$ of free
 $R$-modules of rank $n$. 
 Let $B(GL_n(-)): (\CAlg^{cn})^{op} \to \widehat{\mathcal{S}}$ be a functor given by $R \mapsto B(GL_n(R))$. 
Then, as a functor, $B(GL_n(-)) \simeq (\Mod_{(-)}^{nfree})^{\simeq}$. 
\end{prop}
\begin{proof}
By taking $x = y = R^n$ in Lemma~\ref{aut}, we have a morphism from $B(GL_n(R))
 \to (\Mod_R^{nfree})^{\simeq}$ for each $R \in (\CAlg^{cn})^{op}$. 

By composing a morphism $B(GL_n(R)) \simeq BAut_{\mathbb{S}}(R^n)$
 obtained by Proposition~\ref{GLgp} with $B(Aut_{\mathbb{S}}(R))
 \to (\Mod_R^{nfree})^{\simeq}$ obtained by Lemma~\ref{aut}, we have a morphism $B(GL_n(R))
 \to (\Mod_R^{nfree})^{\simeq}$. 
It suffices to show that a morphism $B(GL_n(R))
 \to (\Mod_R^{nfree})^{\simeq}$ is an equivalence and functorial with
 respect to $R$. 

By the construction, we identify $B(GL_n(R))$ 
with a full $\infty$-subgroupoid of $(\Mod_R^{nfree})^{\simeq}$ such
 that the object is only $R^{\oplus n}$ and the class of morphisms is
 identified with $GL_n(R) \simeq Aut_{\mathbb{S}}(R)$ by
 Proposition~\ref{GLgp}, which is fully faithfull functorial assignment with respect to $R$.

We show that it is
 essentially surjective. It suffices to show that
 $(\Mod_R^{nfree})^{\simeq}$ is connected as a simplicial set. This is
 obvious since, for $M \in \Mod_R^{nfree}$, we have $M \simeq
 R^{\oplus n}$. 
\end{proof}

\subsection{A natural transformation from $B(GL_n(-))$ to $(B^{\mathcal{G}}GL_n)(-)$}
We construct a natural
transformation from $B(GL_n(-))$ to $(B^{\mathcal{G}}GL_n)(-)$ by using the
comparison between the sheaf of projective modules of rank $n$
and the sheafification of the functor of free modules of rank $n$.

For each $R$, we have the following colimit in $\widehat{\mathcal{S}}$:
\[
\xymatrix@1{
 GL_n(R) \ar@<-0.5ex>[r] \ar@<0.5ex>[r]  & \ast \ar[r] & B(GL_n(R)) .
} 
\]
Note that the functor corresponds to this cofiber is not a sheaf since
the second condition of \cite[Definition 2.6.14 (1)]{DAG8}
fails. According to the notation in Definition~\ref{Shff}, we
write $\widetilde{B(GL_n(-))}$ for the sheafification of $B(GL_n(-))$. 
\begin{lemma}
We have the following cofiber sequence
\[
\xymatrix@1{
 GL_n \ar@<-0.5ex>[r] \ar@<0.5ex>[r]  & \ast \ar[r] & \widetilde{B(GL_n(-))}
} 
\]
in $\Shv_{\widehat{\mathcal{S}}}(\CAlg^{\mathcal{G}})$. 
\end{lemma}
\begin{proof}
Since $GL_n$ is corepresented by an $\E$-ring, it is flat sheaf
 by~\cite[VII, Theorem 5.15]{DAG7}, so that it is a Zariski
 (resp. Nisnevich) sheaf. Since sheafification is left
 adjoint, it commutes with the cofiber sequence. 
\end{proof}

For each $A \to R$ in $\CAlg^{cn}$, the natural inclusions induce the
commutative diagram
\[
\xymatrix@1{
 \Mod_A^{nfree} \ar[r] \ar[d] & \ar[d] \Mod_A^{nproj} \\
\Mod_R^{nfree} \ar[r] & \Mod_R^{nproj},
} 
\]
in $\widehat{\Cat{\infty}}$, so that we have a natural transformation
\begin{equation}\label{bpr}
\Mod_{(-)}^{nfree} \to \Mod_{(-)}^{nproj}. 
\end{equation}
By the adjointness of
sheafification, we have a morphism $\widetilde{B(GL_n(-))} \to
(nProj)^{\simeq}$ of $\widehat{\mathcal{S}}$-valued sheaves.

\subsection{An equivalence $B^{\mathcal{G}}GL_n \simeq (nProj)^{\simeq}$}
\begin{prop}\label{bgl}
The  morphism (\ref{bpr}) in the previous subsection gives an equivalence $B^{\mathcal{G}}GL_n \simeq (nProj)^{\simeq}$ in $\Shv_{\widehat{\mathcal{S}}}(\CAlg^{\mathcal{G}})$. 
\end{prop}
\begin{proof}
From adjointness of a morphism $f_{(-)} : \widetilde{B(GL_n(-))} \to
(nProj)^{\simeq}$ of $\widehat{\mathcal{S}}$-valued sheaves, 
we have the following homotopy commutative diagram of $\widehat{\mathcal{S}}$-valued presheaves:
 \[
     \xymatrix@1{
& \widetilde{B(GL_n(-))} \ar[dr]^f & \\
B(GL_n(-)) \ar[ur]^L \ar[rr]^i  && (nProj)^{\simeq} ,
}
\]
where $L$ is the morphism associated to the sheafification and $i$ is
 induced from the inclusion. 

Let us denote the
 limits of the following simplicial diagrams by $B(GL_n(A))_R$ :
\[
\xymatrix@1{
 B(GL_n(A)) \ar@<-0.5ex>[r] \ar@<0.5ex>[r]  & B(GL_n(A
 \otimes_R A)) \ar@<-0.6ex>[r] \ar@<-0.1ex>[r] \ar@<0.5ex>[r]
 & \cdots ,
} 
\]
where $R \to A$ is a faithfully flat morphism. 
By Remark~\ref{6229}, the sheafification $\widetilde{B(GL_n(-))}$ is
 described by the term of $B(GL_n(-))$, i.e., it is the colimit of those
 $B(GL_n(A))_R$ taken
 over the every covering sieve $R \to A$ .

For $f_R:
 \widetilde{B(GL_n(R))} \to (nProj)^{\simeq}(R)$ and an object $Q$ in
 $(nProj)^{\simeq}(R)$, let $(Q)^{\simeq}$ denote the full
 $\infty$-subgroupoid spanned by $Q$. We construct $h_{R,Q} : (Q)^{\simeq} \to \widetilde{B(GL_n(R))}$ as follows.

We take a Zariski local $A$ of $R$ such that an object $Q \in (nProj)^{\simeq}(R)$
 is trivialized on $A$.   
Then, we naturally identified the $\infty$-groupoid spanned by $Q \otimes A$ with $B(GL_n(A))$, so that we have an object in $\widetilde{B(GL_n(R))}$. This is the
 assignment under the morphisms of $\infty$-groupoids.  
By $f \circ L \simeq i$ in the above diagram, we have $f_R \circ h_{R,Q}$
 is homotopic to identity. 

Conversely, for an object $P$ in $\widetilde{B(GL_n(R))}$, we take its
 value on the covering sieve $R \to R$ under the morphisms associated to
 the colimit, and denote by $P'$. Then, by
 an extension of the coefficients of $P'$ gives an object in
 $\widetilde{B(GL_n(R))}$, which is equivalent to $P$ by the
 construction. This shows that $h_{R, f(P)} \circ f_R$ is homotopic to
 identity. 

Since $f_R$ is obviously a Kan fibration and, by the above arguments,
 any fiber of $f_R$ is contractible. By \cite[Lemma 4.1.3.2, Corollary
 4.1.2.6]{HT}, it is a weak homotopy equivalence.  
\end{proof}

\subsection{Zariski connected $\E$-rings}
Next, we consider the condition on a connective $\E$-ring $R$ such that
any finitely generated projective $R$-module has finite constant rank. 

\begin{defn}
\begin{enumerate}[(i)]
\item We say that an $\E$-ring is Zariski non-connected if there exists those
 objects $f, g \in \pi_0 R$ such that $R[f^{-1}] \otimes_R R[g^{-1}]
 \simeq 0$ and $R \simeq R[f^{-1}] \times R[g^{-1}]$. 
\item We say that an $\E$-ring is Zariski connected if it is not Zariski non-connected. 
\end{enumerate}
\end{defn}
\begin{lemma}\label{zcon}
The following conditions are equivalent:
\begin{enumerate}[(i)]
\item $R$ be an $\E$-ring such that $\pi_0 R$ has no non-trivial
      idempotent element. 
\item $R$ is Zariski connected.
\item Any $P \in \Mod_R^{proj}$ has finite constant rank. 
\end{enumerate}
\end{lemma}
\begin{proof}
For proving (iii) from $(i)$, let $R$ be an $\E$-ring such that $\pi_0 R$ has no non-trivial idempotent, and $P \in \Mod_R^{proj}$. Note that $\pi_0 P$ is a finitely generated projective $\pi_0R$-module~\cite[Remark 7.2.2.20]{HA}. Then, there exists $f_1, \cdots, f_m$ which generate the unit ideal of $\pi_0 R$ such that each $\pi_0 P[f^{-1}_i]$ is free of finite rank $n_i$ over $\pi_0 R[f^{-1}_i]$. For an ordinary commutative ring $\pi_0 R$, in this the condition $(i)$ is equivalent to that $n_i$s are constant, and let us denote it by $n$. 
Since finitely generated free $\pi_0R[f^{-1}_i]$-modules $\pi_0
 P[f^{-1}_i]$ can be lifted by finitely generated $R[f^{-1}_i]$-modules
 $P[f^{-1}_i]$ and a morphism between flat $R$-modules is an equivalence
 if and only if it induces an isomorphism on $\pi_0$~\cite[Lemma
 7.2.2.17]{HA}, locally $P$ has a constant rank. 

Conversely, if $\pi_0
 R$ has a non-trivial idempotent element, we have that the rank $n_i$
 are not constant. In this case, the above argument shows that (iii)
 implies (i).    

We show that (i) implies (ii). Since $\pi_0$ preserves finite products, if $R$ is Zariski non-connected, by passing to $\pi_0$ and applying the ordinary commutative ring theory, $\pi_0 R$ has a non-trivial idempotent. 
Conversely, if $\pi_0 R$ has a non-trivial idempotent $e$, it also a
 non-trivial idempotent of $\pi_{\ast}R$. 
Since the localization of an $\E$-ring with one element commutes with
 $\pi_{\ast}$, we have $\pi_{\ast}(R[e^{-1}]) \cong
 (\pi_{\ast}R)[e^{-1}]$ and $\pi_{\ast}(R[(1-e)^{-1}]) \cong
 (\pi_{\ast}R)[(1-e)^{-1}]$. 
Since $\pi_{\ast}$ commutes with finite products, by taking $\pi_{\ast}$
 of $R[e^{-1}] \times R[(1-e)^{-1}]$, we conclude that $R$ is equivalent
 to $R[e^{-1}] \times R[(1-e)^{-1}]$, so that $R$ is Zariski non-connected. 
Thus, (ii) implies (i).  
\end{proof}

\subsection{Comparison between $B^{\mathcal{G}}GL$ and $\Mod_R^{\infty proj}$}

Now, we will prove the main proposition in this section by applying the following proposition. 
\begin{prop}[\cite{DAG7} Lemma 3.21]\label{321}
Let $\mathcal{C}$ be an $\infty$-topos, $I$ an index set and $\{X_i\}_{i
 \in I}$ a collection of objects in $\mathcal{C}$. For every subset $J
 \subset I$, let $X_J \simeq \coprod_{i \in J} X_i$. Let $C \in
 \mathcal{C}$ be an object such that every covering of $C$ has a finite
 subcovering~\cite[Definition 3.1]{DAG7}. Then, the canonical morphism
\[
 \mathop{\mathrm{colim}}_{I \subset I} \Map_{\mathcal{C}}( C, \, X_J)
 \to \Map_{\mathcal{C}}( C, \, X_I)
\]
induces a homotopy equivalence. Here, the colimit in the left hand side
 is run through the all finite subsets $J \subset I$. 
\end{prop}
\qed

\begin{lemma}\label{zcon2}
Let $I \subset \mathbb{N}$ be a finite index set. If $R$ is Zariski connected,
 the canonical morphism
\[
 \coprod_{i \in I} \Map_{\Shv_{\widehat{\mathcal{S}}}(\CAlg^{\mathcal{G}})}(\Spec^{\mathcal{G}} R, \,
 B^{\mathcal{G}}GL_i) \to \Map_{\Shv_{\widehat{\mathcal{S}}}(\CAlg^{\mathcal{G}})}(\Spec^{\mathcal{G}} R, \,
 \coprod_{i \in I}B^{\mathcal{G}}GL_i),
\]
is an equivalence. 
\end{lemma}
\begin{proof}
To check that the canonical morphism is an equivalence, it suffices to 
show that, for any morphism $\phi$ in 
$\Map_{\Shv_{\widehat{\mathcal{S}}}(\CAlg^{\mathcal{G}})}(\Spec^{\mathcal{G}} R, \,
 \coprod_{i \in I}B^{\mathcal{G}}GL_i)$, there exists an index $j \in \mathbb{N}$
 which is uniquely determined by $\phi$, the morphism $\phi$ has the
 unique factorization 
\[
 \Spec^{\mathcal{G}} R \to B^{\mathcal{G}}GL_j \to \coprod_{i \in I}B^{\mathcal{G}}GL_i.  
\]

Since the sheafification commutes with coproducts, $\coprod_{i \in I}B^{\mathcal{G}}GL_i$ is equivalent to the sheafification of the presheaf given by
 $R \mapsto \coprod_{i \in I} (B^{\mathcal{G}}GL_i(R))$. (Note that the coproduct of
 sheaves are the sheafification of the objectwise coproduct as
 presheaves, and an functor obtained by objectwise coproduct of sheaves
 is not always a sheaf.) Therefore, an object of $(\coprod_{i \in I} B^{\mathcal{G}}GL_i)(R)$ can be represented by a finitely generated
 projective $R$-module whose local ranks are contained in $I$.

Let $I\Mod_R^{proj}$ be an $\infty$-category of finitely generated
 projective modules whose local rank is contained in $I$. 
Since $R$ is Zariski connected, by Lemma~\ref{zcon}(iii), we deduce that
 $I\Mod_R^{proj} \simeq \coprod_{\{ i \in I \}} \Mod_R^{iProj}$. 
Therefore, for any object of $(\coprod_{i \in I} B^{\mathcal{G}}GL_i)(R)$, there exists
 a unique $j \in I$ such that the object can be represented by a finitely
 generated projective $R$-module of rank $j$. 

By applying the Yoneda embedding~\cite[Section 5.1.3]{HT}, 
an object of $(\coprod_{i \in I} B^{\mathcal{G}}GL_i)(R)$ represented by a finitely
 generated projective $R$-module of rank $j$ corresponds to the morphism
 $\Spec^{\mathcal{G}} R \to \coprod_{i \in I}B^{\mathcal{G}}GL_i$ which has the unique factorization 
\[
 \Spec^{\mathcal{G}} R \to B^{\mathcal{G}}GL_j \to \coprod_{i \in I}B^{\mathcal{G}}GL_i.  
\]
\end{proof}

We consider a decomposition of an $\E$-ring $R$ by using the idempotent element in $\pi_0 R$. 

\begin{prop}\label{qclf}
Let $\Spec^{\mathcal{G}} R : (\CAlg^{\mathcal{G}})^{op} \to \widehat{\mathcal{S}}$ be a spectral scheme and $\Spec^{\mathcal{G}} R \to
 B^{\mathcal{G}}GL_n$ a morphism in
 $\Shv_{\widehat{\mathcal{S}}}(\CAlg^{\mathcal{G}})$. Let $B^{\mathcal{G}}GL$ be a sheaf $\coprod_{n \in \mathbb{N}} B^{\mathcal{G}}GL_n$. 
 For $I \subset J \subset \mathbb{N}$, we have the system $\coprod_{i \in I} B^{\mathcal{G}}GL_i \to \coprod_{j \in J} B^{\mathcal{G}}GL_j$ given by inclusions. 
 Then there is an
 equivalence of $\infty$-groupoids; 
\[
 \Map_{\Shv_{\widehat{\mathcal{S}}}(\CAlg^{\mathcal{G}})}(\Spec^{\mathcal{G}} R, \,
 B^{\mathcal{G}}GL) \simeq (\Mod_R^{\infty proj})^{\simeq}. 
\]
\end{prop}
\begin{proof}
By Proposition~\ref{321}, we have 
\[
 \mathop{\mathrm{colim}}_{I \subset \mathbb{N}} \Map_{\Shv_{\widehat{\mathcal{S}}}(\CAlg^{\mathcal{G}})}(\Spec^{\mathcal{G}} R, \, \coprod_{i \in I}B^{\mathcal{G}}GL_i) \simeq \Map_{\Shv_{\widehat{\mathcal{S}}}(\CAlg^{\mathcal{G}})}(\Spec^{\mathcal{G}} R, \, B^{\mathcal{G}}GL). 
\]

By decomposing $R$ with the Zariski connected
 $\E$-rings (given by the corresponding irreducible decomposition on
 $\pi_0$) and applying Lemma~\ref{zcon2}, for $I=\{1, \cdots, n\}$, we
 have the equivalence
\[
 \mathop{\mathrm{colim}}_{I \subset \mathbb{N}} \coprod_{i \in I} \Map_{\Shv_{\widehat{\mathcal{S}}}(\CAlg^{\mathcal{G}})}(\Spec^{\mathcal{G}} R, \,
 B^{\mathcal{G}}GL_i) \simeq \mathop{\mathrm{colim}}_{I \subset \mathbb{N}}
 \Map_{\Shv_{\widehat{\mathcal{S}}}(\CAlg^{\mathcal{G}})}(\Spec^{\mathcal{G}} R, \, \coprod_{i \in I}B^{\mathcal{G}}GL_i).  
\]
By Proposition~\ref{bgl}, we identify the left hand side with
\[
 \mathop{\mathrm{colim}}_{i \in \mathbb{N}}(\Mod^{iproj}_R)^{\simeq}
 \simeq (\Mod_R^{\infty Proj})^{\simeq}. 
\]
\end{proof}

\section{$K$-theory of $\Mod_R^{\infty proj}$}

\begin{defn}[cf. \cite{ba} 1.2,  \cite{BGT-End} Definition 2.1]\label{wcof}
Let $\mathcal{C}$ be a pointed $\infty$-category. A class of $w^{\infty}$-cofibrations is a class of morphisms in $\mathcal{C}$ which satisfies the following conditions:
\begin{enumerate}[(i)]
\item $* \to X$ is a $w^{\infty}$-cofibration for any object $X$,
\item The class of $w^{\infty}$-cofibrations includes weak equivalences, 
\item Any composition of $w^{\infty}$-cofibrations is a $w^{\infty}$-cofibration,  
\item For a $w^{\infty}$-cofibration $X \to Y$ and a morphism $X \to Z$, there
      exists a pushout square
\[ 
 \xymatrix@1{
X \ar[r] \ar[d] & Y \ar[d] \\
 Z \ar[r]_f & W,
} 
\]
in which the morphism $f$ is a $w^{\infty}$-cofibration. 
\end{enumerate}
We call such a pair of $\mathcal{C}$ and a $w^{\infty}$-cofibrations an $\infty$-category with $w^{\infty}$-cofibrations. 
\end{defn}
In terminology of \cite{BGT}, a $w^{\infty}$-cofibration is called a cofibration. In terminology  of \cite{ba1}, a $w^{\infty}$-cofibration is called an igressive morphism. 
\begin{rem}\label{spplit}
We say that a $w^{\infty}$-cofibration is a split
 $w^{\infty}$-cofibration if it is split as a morphism.  
\end{rem}


\begin{defn}\label{ssplit}
We make into $\Mod_R^{\infty proj}$ to be an
 $\infty$-category with $w^{\infty}$-cofibrations as follows. 

Declare a morphism $f: P_1 \to P_2$ in
$\Mod_R^{\infty proj}$ to be a $w^{\infty}$-cofibration if it is a morphism in $\Mod_R$ and the cofiber of $f$ is an object in
$\Mod_R^{\infty proj}$. 

Then, $\Mod_R^{proj}$ (resp. $\Mod_R^{\infty proj}$) is an
 $\infty$-category with $w^{\infty}$-cofibrations. 
\end{defn}

\begin{lemma}\label{ssplit2}
\begin{enumerate}[(i)]
\item A $w^{\infty}$-cofibration in $\Mod_R^{\infty proj}$ is always split, e.g., it is a
      split $w^{\infty}$-cofibration in Remark~\ref{spplit}. 
\item Let $R_1 \to R_2$ be a morphism of connective
$\E$-rings. Then, a functor $\Mod_{R_1}^{\infty proj} \to
\Mod_{R_2}^{\infty proj}$ given by the extension of coefficients is an exact functor in the sence of $K$-theory. 
\end{enumerate}
\end{lemma}
\begin{proof}
By \cite[Proposition 7.2.2.6 (5)]{HA}, a cofiber sequence of finitely
 generated projective modules is always split up to homotopy. Thus, (i)
 holds. 

Since the tensor product commutes with the cofiber sequences and
 preserves rank, the extension of coefficients is exact.  
\end{proof}

%

\begin{rem}
In this paper, we consider the connective $K$-theory. We can regard the $K$-theory spectrum as a $K$-theory space.  
\end{rem}

Let $\mathcal{C}$ be a pointed additive $\infty$-category. 
We recall an $\infty$-version of the additivity
theorem~\cite[Theorem 1.8.7]{ww} given by \cite{FLP}. We state the
theorem as in Lurie's unpublished note~\cite{Lurielec}. 

Recall the notion of a split
$w^{\infty}$-cofibration from Remark~\ref{spplit} and the notion of the
group completion. 
We will apply the following theorem for $\Mod_R^{\infty
proj}$ whose $w^{\infty}$-cofibrations defined in Definition~\ref{ssplit} are split
$w^{\infty}$-cofibrations by Lemma~\ref{ssplit2}. 

\begin{thm}[\cite{FLP}, cf. \cite{Lurielec} Theorem 10]\label{addit}
Let $\mathcal{C}$ be a pointed additive $\infty$-category with split $w^{\infty}$-cofibrations, and $\mathcal{C}^{\simeq}$ be the maximal $\infty$-groupoid of $\mathcal{C}$ as in (\ref{shf}). Let $K(\mathcal{C})$
 be the algebraic $K$-theory and $\Omega B(\mathcal{C}^{\simeq})$ the group
 completion of $\mathcal{C}^{\simeq}$. 

Then, there is an equivalence 
\begin{equation}\label{luria}
 \Omega B (\mathcal{C}^{\simeq}) \to
 K(\mathcal{C})
\end{equation}
 of spases. Here, $\Omega$ and $B$ are defined in Section $4$.  
\end{thm}
\begin{cor}\label{add}
We have $\Omega B ((\Mod_R^{\infty proj})^{\simeq}) \simeq
 K(\Mod_R^{\infty proj})$. 
\end{cor}
\begin{proof}
The $\infty$-category $\Mod_R^{\infty proj}$
 (resp. $\Mod_R^{proj}$) is additive, pointed by $0$,
 with split $w^{\infty}$-cofibrations. Thus, Theorem~\ref{addit} can be applied. 
\end{proof}
\section{Proof of Theorem~\ref{main1}}
Let us keep the notation explained in the previous sections. 
Now we prove Theorem~\ref{main1}. 

\begin{defn}\label{Bgl0}
Let $B^{\mathcal{G}}GL = \coprod_{n \in \mathbb{N}} B^{\mathcal{G}}GL_n$. 
We define a morphism $t: B^{\mathcal{G}}GL \times B^{\mathcal{G}}GL \to B^{\mathcal{G}}GL$ by the morphism induced from the assignment
\[
 B^{\mathcal{G}}GL_n \times B^{\mathcal{G}}GL_m \to B^{\mathcal{G}}GL_{n+m}
\]
 which sends a pair
 $(P, Q)$ of projective modules of rank $n$ and $m$ respectively to the
 projective module $P \oplus Q$ of rank $n + m$ under the identification
 of Proposition~\ref{bgl}. 
\end{defn}
\begin{prop}\label{Bgl}
Let $B^{\mathcal{G}}GL = \coprod_{n \in \mathbb{N}} B^{\mathcal{G}}GL_n$. 
Then, $B^{\mathcal{G}}GL$ becomes a commutative $\infty$-monoid by the morphism $t$ defined in Definition~\ref{Bgl0} as the multiplication. 
\end{prop}
\begin{proof}
For projective modules $P$ and $Q$ of rank $n$ and $m$ respectively, $P \oplus Q$ is the projective module of rank $n + m$. This assighment is obviously commutative and associative. 
\end{proof}
\begin{thm}\label{m1}
Let $\CAlg^{\mathcal{G}}$ be the $\infty$-category $\CAlg^{cn}$ equipped
 with Zariski or Nisnevich topology defined in Section $2$. 
Let $B^{\mathcal{G}}GL = \coprod_{n \in \mathbb{N}} B^{\mathcal{G}}GL_n$
 and $\Omega B^{\mathcal{G}} (B^{\mathcal{G}}GL)$ the group completion as a sheaf on $\CAlg^{\mathcal{G}}$. 
Note that $B^{\mathcal{G}}GL$ becomes a commutative $\infty$-monoid as in Proposition~\ref{Bgl}. 

There is an equivalence of $\infty$-groupoids:
\[
 \Map_{\Shv_{\widehat{\mathcal{S}}}(\CAlg^{\mathcal{G}})}(\Spec^{\mathcal{G}} R, \, \Omega B^{\mathcal{G}}  (B^{\mathcal{G}}GL)) = \widetilde{K}^{\mathcal{G}}(\Mod_R^{\infty proj}), 
\]
where $\widetilde{K}$ is the sheafification of $K$ defined by (\ref{func}). 
\end{thm}
\begin{proof} 
By Proposition~\ref{bgl}, we have an equivalence 
\[
 \Omega B (B^{\mathcal{G}}GL(R)) \simeq \Omega B ((\Mod_{R}^{\infty proj})^{\simeq}),
\]
where $(-)^{\simeq}$ denotes the maximal $\infty$-groupoid~(\ref{shf}). 
Since all $w^{\infty}$-cofibrations in $\Mod_R^{\infty proj}$ are split and the
 homotopy category of $\Mod^{proj}_R$ is additive, by
 Corollary~\ref{add}, we obtain that $\Omega B((\Mod_{R}^{\infty
 proj})^{\simeq})$ is equivalent to the algebraic $K$-theory
$K(\Mod_R^{\infty proj})$ given by $S_{\bullet}$ construction. 

On the other hand, we have an equivalence induced by Yoneda embedding 
\[
 \Map_{\Shv_{\widehat{\mathcal{S}}}(\CAlg^{\mathcal{G}})}(\Spec^{\mathcal{G}} R, \, \Omega B^{\mathcal{G}} (B^{\mathcal{G}}GL)) \simeq (\Omega B ^{\mathcal{G}} B^{\mathcal{G}}GL)(R). 
\]
By virtue of Proposition~\ref{Bgl}, we can apply Proposition~\ref{gpcmp}
 to the commutative $\infty$-monoid sheaf $B^{\mathcal{G}}GL$. The sheafification of the objectwise group completion functor
 $R \mapsto \Omega B(B^{\mathcal{G}}GL(R))$ is equivalent to the group completion of the sheaf given
by the assignment $R \mapsto (\Omega B^{\mathcal{G}}  B^{\mathcal{G}}GL)(R)$. Thus, we
 have an equivalence $(\Omega B^{\mathcal{G}}  B^{\mathcal{G}}GL)(R) \simeq \Omega B (B^{\mathcal{G}}GL(R))$ after the sheafification.   
\end{proof}

\bibliographystyle{amsplain} \ifx\undefined\bysame

\begin{bibdiv}
\begin{biblist}

\bib{ba1}{article}{
   author={Barwick, C.},
   title={On the algebraic $K$-theory of higher categories},
   journal={Preprint, available at arxiv:1204.3607}
   date={2014}
}

\bib{ba}{article}{
   author={Barwick, C.},
   title={On the exact $\infty$-categories and the theorem of the heart},
   journal={Preprint, available at arxiv:1212.5232v4}
   date={2014}
}

\bib{BGT}{article}{
   author={Blumberg, Andrew J.},
   author={Gepner, David},
   author={Tabuada, Gon{\c{c}}alo},
   title={A universal characterization of higher algebraic $K$-theory},
   journal={Geom. Topol.},
   volume={17},
   date={2013},
   number={2},
   pages={733--838},
   issn={1465-3060},
   review={\MR{3070515}},
   doi={10.2140/gt.2013.17.733},
}

\bib{BGT-End}{article}{
   author={Blumberg, A.},
author={Gepner, D.},
author={Tabuada, G.},
   title={The algebraic $K$-theory endomorphism},
   journal={available at arxiv:1302.1214}
   date={2014}
}

\bib{FLP}{article}{
   author={Fiore, T.},
author={Luck, W.},
author={Pieper, M.},
   title={Waldhausen Additivity : Classical and Quasicategorical},
   journal={Preprint, available at arXiv:1207.6613}
   date={2012}
}


\bib{GGN}{article}{
   author={Gepner, D.},
author={Groth, M.},
author={Nikolaus, T.},
   title={Universality of multiplicative infinite loop space machines},
   journal={Preprint, available at arXiv:1305.4550v1}
   date={2013}
}

\bib{kodsag}{article}{
   author={Kodjabachev, D.},
   author={Sagave, S.},
   title={Strictly commutative models for $\E$ quasi-categories},
   journal={Preprint, available at arXiv:1408.0116v2}
   date={2014}
}


\bib{HT}{book}{
   author={Lurie, J.},
   title={Higher topos theory},
   series={Annals of Mathematics studies},
   volume={170},
   publisher={Princeton University Press},
   date={2009}
   pages={xv+925},
}

\bib{HA}{article}{
   author={Lurie, J.},
   title={Higher algebra},
   journal={Preprint, available at www.math.harvard.edu/lurie}
   date={2011}
}

\bib{DAG8}{article}{
   author={Lurie, J.},
   title={Quasi-Coherent Sheaves and Tannaka Duality Theorems},
   journal={Preprint, available at www.math.harvard.edu/lurie}
   date={2011}
}

\bib{DAG11}{article}{
   author={Lurie, J.},
   title={Decsent theorems},
   journal={Preprint, available at www.math.harvard.edu/lurie}
   date={2011}
}

\bib{DAG5}{article}{
   author={Lurie, J.},
   title={Structured spaces},
   journal={Preprint, available at www.math.harvard.edu/lurie}
   date={2011}
}
\bib{DAG7}{article}{
   author={Lurie, J.},
   title={Spectral Schemes},
   journal={Preprint, available at www.math.harvard.edu/lurie}
   date={2011}
}

\bib{Lurielec}{article}{
   author={Lurie, J.},
   title={Additive $K$-Theory (Lecture 18)},
   journal={lecture available at http://www.math.harvard.edu/~lurie/281notes/Lecture18-Rings.pdf},
   date={2014},
}

\bib{Lurielec19}{article}{
   author={Lurie, J.},
   title={Algebraic $K$-Theory of Ring Spectra (Lecture 19)},
   journal={available at http://www.math.harvard.edu/
~lurie/281notes/Lecture19-Rings.pdf},
   date={2014},
}



\bib{MV}{article}{
    author={Morel, F.},
    author={Voevodsky, V.},
    title={{${\bf A}\sp 1$}-homotopy theory of schemes},
    journal={Institut des Hautes \'Etudes Scientifiques. Publications
              Math\'ematiques},
    volume={90},
    date={1999},
    pages={45--143 (2001)}
 }  




\bib{ww}{article}{
   author={Waldhausen, F.},
   title={Algebraic $K$-theory of spaces},
      journal={Lecture Notes in Math.},
      volume={1126},
      publisher={Springer, Berlin},
   date={1985},
   pages={318--419},
}
 
\end{biblist}		
\end{bibdiv}
\newcommand{\bysame}{\leavemode\hbox to3em{\hrulefill}\,} \fi

\end{document}